\documentclass{article}

\usepackage{authblk}
\usepackage{hyperref}
\usepackage{amsmath}
\usepackage{amsthm}
\usepackage{amssymb}
\usepackage{color}
\usepackage{tikz}
\usepackage{xcolor}
\usepackage{tkz-graph}
\usepackage{comment}
\usepackage{graphicx,float}
\usepackage{graphics}
\usepackage{enumitem}
\usepackage{mathtools}
\usepackage{multirow}
\usepackage[mathlines]{lineno}
\usepackage[normalem]{ulem}
\usepackage{ulem}  

\newtheorem{theorem}{Theorem}[section]
\newtheorem{lemma}[theorem]{Lemma}

\newtheorem{proposition}[theorem]{Proposition}
\newtheorem{remark}[theorem]{Remark}
\newtheorem{definition}[theorem]{Definition}

\newcommand{\be}{\begin{equation}}
\newcommand{\ee}{\end{equation}}

\title{Graph transformations and eigenvectors of the Laplacian, signless Laplacian and Adjacency matrices}

\author[1]{\normalsize{Jean-Guy Caputo }\thanks{jean-guy.caputo@insa-rouen.fr}}
\author[2]{\normalsize{Leonardo de Lima} \thanks{leonardo.delima@ufpr.br}}
\author[1]{\normalsize{Arnaud Knippel} \thanks{arnaud.knippel@insa-rouen.fr}}

\affil[1]{Laboratoire de Math\'ematiques, INSA Rouen Normandie, 76800 Saint-Etienne du Rouvray, France}
\affil[2]{Programa de P\'os-Gradua\c c\~ao em Matem\'atica, Universidade Federal do Paran\'a, Brazil}

\date{\ }

\begin{document}
\maketitle

\begin{abstract}
We study graph transformations and their effects on the eigenvectors of three fundamental graph matrices: the adjacency matrix $A$, the signless Laplacian $Q$, and the Laplacian $L$. We extend the Laplacian eigenvector edge and matching principles introduced by Merris to $Q$ and $A$. We also extend some of the known $L$ graph transformations to $Q$ and $A$ and prove that graphs admitting a common eigenvector $\mathbf{x}$ for $A, L$ and $Q$ are
such that vertices corresponding to nonzero components of $\mathbf{x}$ have same degree. Following our previous work for $L$, we investigate bivalent graphs for $A$ and $Q$, that is graphs admitting an eigenvector with all components in $\{-1, 1\}$. In particular, for trees, we characterize $Q$-bivalence and establish that $A$-bivalence forces all vertex degrees to be odd. 
\end{abstract}

{\bf MSC classification 2020:} 05C50  ,  15A18

\tableofcontents

\section{Introduction}\label{intro}

Given a finite graph, important operators are the adjacency matrix $A$,
the graph Laplacian $L= D-A$ where $D$ is the { diagonal matrix of the vertex degrees of the graph} and the 
signless Laplacian $Q=D+A$.
The graph Laplacian is very important for applications in engineering
description of evolution problems
on graphs; an example is the graph wave equation, a natural model for 
miscible flows on a network, see the articles \cite{maas,cks13}.
This simple formalism proved very useful for modeling the electrical grid \cite{ckr19} 
or describing an epidemic on a geographical network \cite{ckmk19}.

The adjacency matrix has been historically more studied for its spectral properties. 
It also has many applications in physics, chemistry and computer science, see for example \cite{crs01}.
The signless Laplacian $Q$ has been the least studied for its spectral properties. It has recently attracted interest as an operator allowing faster convergence in quantum walks than the 
Laplacian \cite{mw25}.

Most results on eigenvalues/eigenvectors for $A,L,Q$
are inequalities, for the Laplacian see for example the { extensive} review by
Mohar \cite{mohar92} and the extensive monograph \cite{bls07}. The book \cite{crs01} studies in detail the spectrum of $A$ and the article \cite{CRS07} analyzes the one of $Q$.
A breakthrough occured in 1998 when Merris \cite{merris}, in a pioneering article,
introduced "Laplacian eigenvector principles" that  predict how the
spectrum of the Laplacian is affected by graph transformations (adding or deleting edges, coalescing vertices, etc). This opens the way to calculating eigenvalues and eigenvectors combinatorially rather than numerically, see  \cite{ck23} for a review of this approach.

In this article, we continue the work of Merris by extending some of his principles to the matrices $Q$ and $A$. We show that a number of principles still hold for $Q$ and $A$. {  We also examine graphs that have eigenvectors in the set $\{-1,1\}$ for the operators $L$, $Q$ or $A$, the so-called {\it $L$-bivalent graphs, $Q$-bivalent graphs or $A$-bivalent graphs}}.\\ 
{
The article is organized as follows.  We extend in Section 2 the results of Merris to $Q$ and $A$. Section 3 examines the case of bivalent graphs, that have an eigenvector of either $A$, $L$ or $Q$ with entries in $\{-1,1 \}$. Finally, in Section 4, we study the special case of bivalent trees. 
}

We introduce the following notations.
Let $G=(V,E)$ be a connected graph with vertex set $V=\{1,2,\ldots,n\}.$  The diagonal degree matrix is $D=diag(d_{1},\ldots,d_{n})$ with entries equal 
to the degree $d_i$ of each vertex $i \in V$.
The adjacency $A$ is such that { $A_{ij}=A_{ji}=1$} if $i \sim j$, and $A_{ij} = 0$ otherwise. 

Let $\mathbf{x} \in \mathbb{R}^{n}$. The quadratic form of $Q$ is given by 
\begin{equation}\label{qq}
    \mathbf{x^{T}} \, Q \, \mathbf{x} = \sum_{ij \in E} (x_i + x_j)^2.
\end{equation}
The matrix $Q$ is symmetric and from \eqref{qq}, note that it is semidefinite 
positive.

The quadratic form associated to the graph Laplacian is
\begin{equation}\label{ql}
    \mathbf{x^{T}} \, L \, \mathbf{x} = \sum_{ij \in E} (x_i - x_j)^2.
\end{equation}
Note the similarity between the quadratic forms of $L$ and $Q$. On the other hand, the quadratic form associated with the adjacency matrix $A$ is
\begin{equation}\label{qa}
    \mathbf{x^{T}} \, A \, \mathbf{x} = \sum_{ij \in E} 2 x_i x_j.
\end{equation}

Finally we recall the following notion of soft node \cite{cks13}.
\begin{definition}
A node $j$ of a graph is a soft node for an eigenvalue
$\lambda$ of the graph Laplacian if there exists an eigenvector
$\mathbf{v}$ for this eigenvalue such that $v_j = 0$.
\end{definition}

\section{Extension of Merris's principles to Q and A}

\subsection{Edge principle}

From Merris's foundational article \cite{merris}, we have the following {\it edge principle} for $L$.
\begin{theorem} 
Let $\lambda$ be an eigenvalue of $G$ for eigenvector $\mathbf{x}$ { for the matrix $L$.} If $x_i = x_j$, 
then $\lambda$ is an eigenvalue of $G^{\prime}$ for $\mathbf{x}$, where $G^{\prime}$ is
the graph obtained from G by deleting or adding $e = \{i, j\}$ 
depending on whether or not it is an edge of G.
\end{theorem}

The quadratic form of $Q$ in \eqref{qq} gives the intuition of a similar edge principle. We prove the following result using the eigenvalue equation.

\begin{theorem} 
Let $\lambda$ be an eigenvalue of G for eigenvector $\mathbf{x}$ for the matrix 
$Q$. If $x_i = -x_j$, 
then $\lambda $ is an eigenvalue of $G^{\prime}$ for $\mathbf{x}$, where $G^{\prime}$ is
the graph obtained from G by deleting or adding $e = \{i, j\}$ 
depending on whether or not it is an edge of G.
\end{theorem}

\begin{proof}
For a vertex $i$ we have 
$$d_i x_i + \sum_{k \sim i} x_k = \lambda x_i,$$
where $d_i$ is the degree of $i$.
Assume there is a vertex $j$ such that $x_i = -x_j$ and let us add the edge
$e = \{i, j\}$ to $G$ to form $G^{\prime}$. Then for $G'$ we have
$$(d_i+1) x_i + \sum_{k \sim i} x_k - x_i = d_i x_i + \sum_{k \sim i} x_k=\lambda x_i,$$
and the proof is complete.
\end{proof}

For matrix $A$, the quadratic form in \eqref{qa} seems to indicate that there is an edge principle when we have $x_i=0$ for some vertex $i$. The theorem below shows that the correct result is a bit more restrictive.

\begin{theorem} 
Let $\lambda$ be an eigenvalue of G for eigenvector $\mathbf{x}$ for the matrix 
$A$ and let two vertices $i$ and $j$ such that $x_i = x_j =0$, 
then $\lambda $ is an eigenvalue of $G^{\prime}$ for $\mathbf{x}$, where $G^{\prime}$ is
the graph obtained from G by deleting or adding $e = \{i, j\}$ 
depending on whether or not it is an edge of G.
\end{theorem}

\begin{proof}
Without loss of generality, consider two nonadjacent vertices $i$ and $j$.  
For $i$ and $j$ we have 
$$ \sum_{k \sim i} x_k = \lambda x_i ,$$
and 
$$ \sum_{k \sim j} x_k = \lambda x_j.$$
Adding an edge $ij$, if we still have $\mathbf{x}$ as an eigenvector of the new matrix for the same eigenvalue, then the eigenvalue equations read:
$$ x_j + \sum_{k \sim i, k \neq j} x_k = \lambda x_i,$$
$$ x_i + \sum_{k \sim j, k \neq i} x_k = \lambda x_j,$$
This holds if and only if both $x_i=0$ and $x_j=0$. In this case, the eigenvalue equations for the other vertices are unaffected.

\end{proof}

\subsection{Alternating (matching) principle}

From Merris \cite{merris}, we have the following result based on a matching on the vertices of the graph. 
\begin{theorem}
Let $\mathbf{x}$ be an eigenvector of G affording $\lambda$ { for the matrix $L.$}
Suppose the vertices in $Z = \{k \in V(G): x(k) \neq 0\}$ can be paired 
up in such a way that if i and j are paired vertices then $x_i = -x_j$. 
Suppose further that all paired vertices in Z are adjacent (not adjacent) to 
each other. Let $G^{\prime}$ be the graph
obtained from G by removing (resp. adding new) edges between all paired vertices of
Z. Then $\mathbf{x}$  is an eigenvector of $G^{\prime}$ for $\lambda - 2$ 
(resp. for $\lambda + 2$).
\end{theorem}

{
The previous result for $L$ can be adapted to $Q$. We have
\begin{theorem}
Let $\mathbf{x}$ be an eigenvector of G affording $\lambda$ {for the matrix $Q$.}
Suppose the vertices in $Z = \{k \in V(G): x(k) \neq 0\}$ can be paired 
up in such a way that if i and j are paired vertices then $x_i = x_j$. 
Suppose further that all paired vertices in Z are adjacent (not adjacent) to 
each other. Let $G^{\prime}$ be the graph
obtained from G by removing (resp. adding new) edges between all paired vertices of
Z. Then $\mathbf{x}$  is an eigenvector of $G^{\prime}$ for $\lambda - 2$ 
(resp. for $\lambda + 2$).
\end{theorem}
}
\begin{proof}
We consider adding a matching; the proof when deleting a matching is similar.
    Let vertex $i \in Z,$ and the corresponding pair $j$ such that $x_{i} = x_{j} \ne 0.$ The Laplacian eigenequation to vertex $i$ in $G^{\prime}$ given by 
\begin{eqnarray}
(d_i+1) x_{i} +\sum_{ \{k,i\} \, \in \,  E(G^{\prime}) } x_{k}  &=& d_i x_{i} + x_{i} + \sum_{\{k,i\} \in E(G) } x_{k} + x_{j} \nonumber \\
&=& \lambda \, x_{i} + x_{i} + x_{j}.
\end{eqnarray}
Since $x_{i} = x_{j},$ then 
$$
(d_i+1) x_{i}+ \sum_{\{i,j\} \in E(G^{\prime}) } x_{j}  = (\lambda+2) \, x_{i},
$$
and the proof is complete.
\end{proof}

Since $A=\frac{Q-L}{2}$, we can deduce a matching principle for $A$ obtained when pairing vertices all with value $0$. However, this is the same as using the edge principle for $A$ described above.

\subsection{Connection of a soft node to all vertices of a graph}

An important geometric transformation implicitly contained in Merris's article is to introduce a new vertex (soft node) connected to all vertices of a graph $G$. For $L$ and $Q$ we have the following.
\begin{theorem}\label{addQ}
Let $G=(V,E)$ be a graph affording an 
eigenvalue $\lambda \neq 0$ for $L$ or $Q$ for an eigenvector { $\mathbf{x}$ with $\sum_i x_i = 0$.} Then 
the new graph $G'$ obtained by adding
a node connected to all the nodes of $G$ has eigenvalue $\lambda +1$
for the eigenvector $\mathbf{x}^{\prime}$ obtained by extending $\mathbf{x}$ by a zero component.
\end{theorem}

\begin{proof}
{ The eigenvalue system reads $Q \mathbf{x} = \lambda \mathbf{x}$. Adding
a vertex $n+1$ of value $0$ connected to all other vertices we obtain 
the following identity.
\be\label{esQ}
\begin{pmatrix} Q+ I_n & \mathbf{1} \\ \mathbf{1}^T & n \end{pmatrix}  
\begin{pmatrix} \mathbf{x} \\ 0 \end{pmatrix}
= \begin{pmatrix} (\lambda+1) \mathbf{x} \\ \mathbf{1}^T \mathbf{x} \end{pmatrix},
\ee
where $I_n$ is the $n \times n$ identity matrix and $\mathbf{1}$ is
the vector of ones of dimension $n$.
	Then, if $\mathbf{1}^T \mathbf{x}= \sum_i x_i = 0$ the new eigenvalue is $\lambda' = \lambda +1$. }
\end{proof}

Figures \ref{fig:graphoperation1} and \ref{fig:graphoperation2} illustrate the operation of Theorem \ref{addQ}.
\begin{figure}[h]
    \centering
\includegraphics[width=0.7\linewidth]{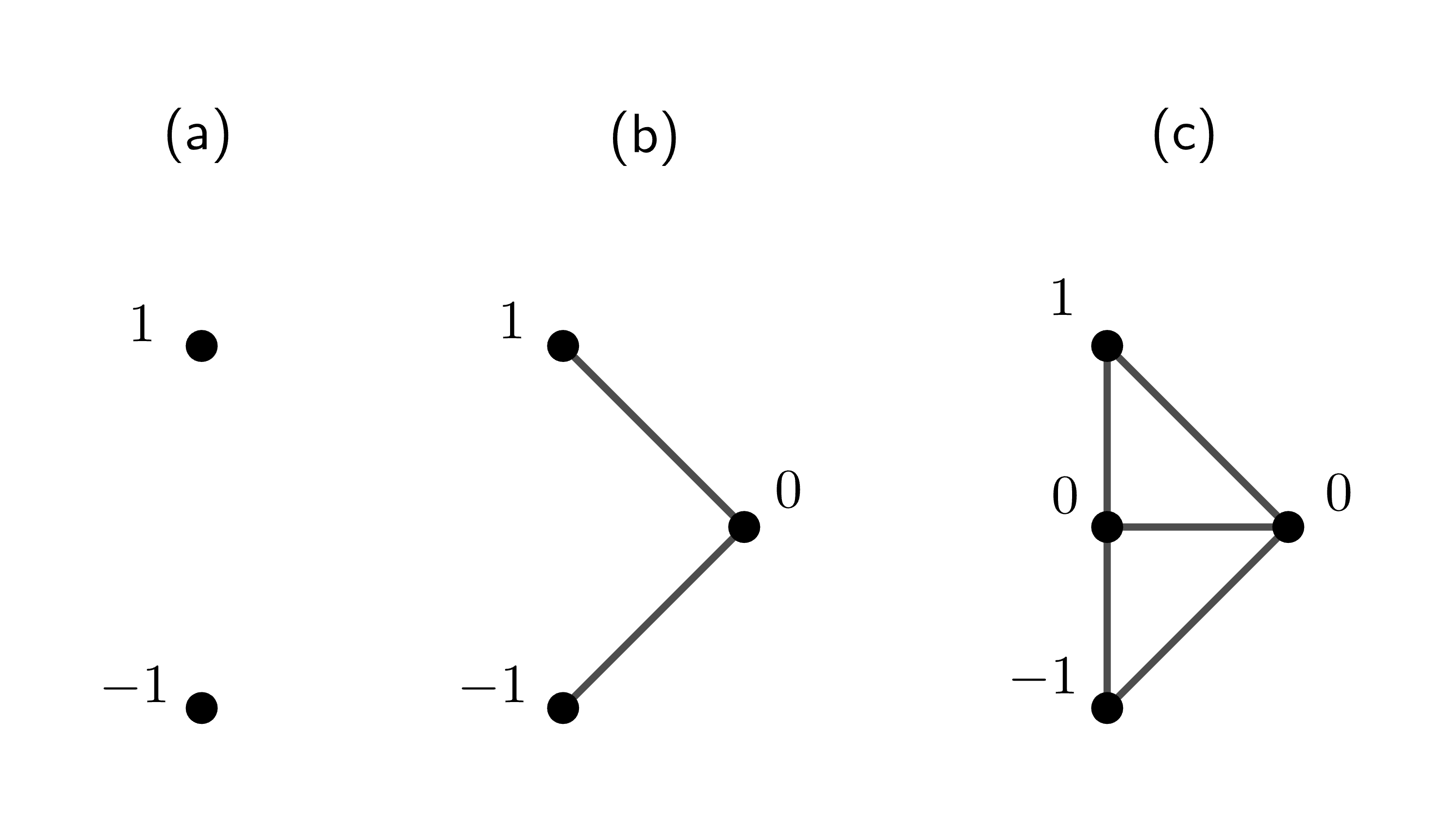}
    \caption{Applying successively the transformation of Theorem \ref{addQ} { to matrix $Q$} in graph $(a)$ ($\lambda = 0$) gives graph $(b)$ ($\lambda = 1$) and graph $(c)$ ($\lambda = 2$).  }
    \label{fig:graphoperation1}
\end{figure}

\begin{figure}[h]
    \centering
\includegraphics[width=0.7\linewidth]{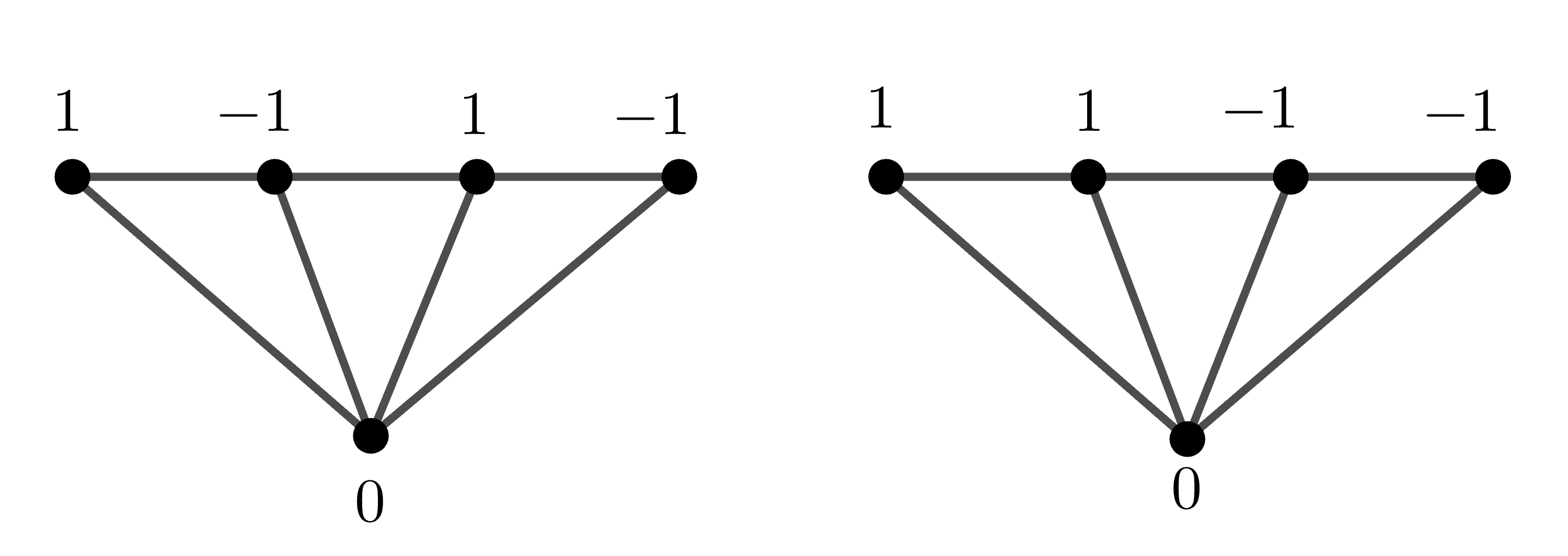}
    \caption{Graph obtained by adding a $0$ vertex to $P_4$ for $\lambda=0$ (left) and $\lambda=2$ (right) { associated with matrix $Q.$}  }
    \label{fig:graphoperation2}
\end{figure}

{ When applying these ideas to $A$, we have a slightly different result.}
\begin{theorem}\label{addA}
Let $G=(V,E)$ be a graph affording an eigenvalue $\lambda \neq 0$ for $A$ { associated with} an eigenvector $\mathbf{x}$ and $W \subset V$ such that $\sum_{i\in W} x_i = 0$.  Then the new graph $G'$ obtained by adding a node connected to all the nodes of $W$ has the same eigenvalue $\lambda$
for the eigenvector $\mathbf{x}^{\prime}$ obtained by extending $\mathbf{x}$ by a zero component.
\end{theorem}
\begin{proof}
The eigenvalue system reads $A \mathbf{x} = \lambda \mathbf{x}$. Adding
a vertex $n+1$ of value $0$ connected to all the vertices in $W$ we obtain
the following identity.
\be\label{esA}
\begin{pmatrix} A & \mathbf{C}_{W} \\ \mathbf{C}_{W}^T & 0 \end{pmatrix}
\begin{pmatrix} \mathbf{x} \\ 0 \end{pmatrix}
= \begin{pmatrix} \lambda \mathbf{x} \\ \mathbf{C}_{W}^T \mathbf{x} \end{pmatrix}
= \lambda \begin{pmatrix}  \mathbf{x} \\ 0\end{pmatrix}.
\ee
where $\mathbf{C}_{W}$ is the vector of dimension $n$ with entries $1$ for vertices
in $W$ and $0$ otherwise.
This shows the result.


\end{proof}

Figures \ref{fig:graphoperation3} and \ref{fig:graphoperation4} illustrate the operation of Theorem \ref{addA}.
\begin{figure}[h]
    \centering
\includegraphics[width=0.7\linewidth]{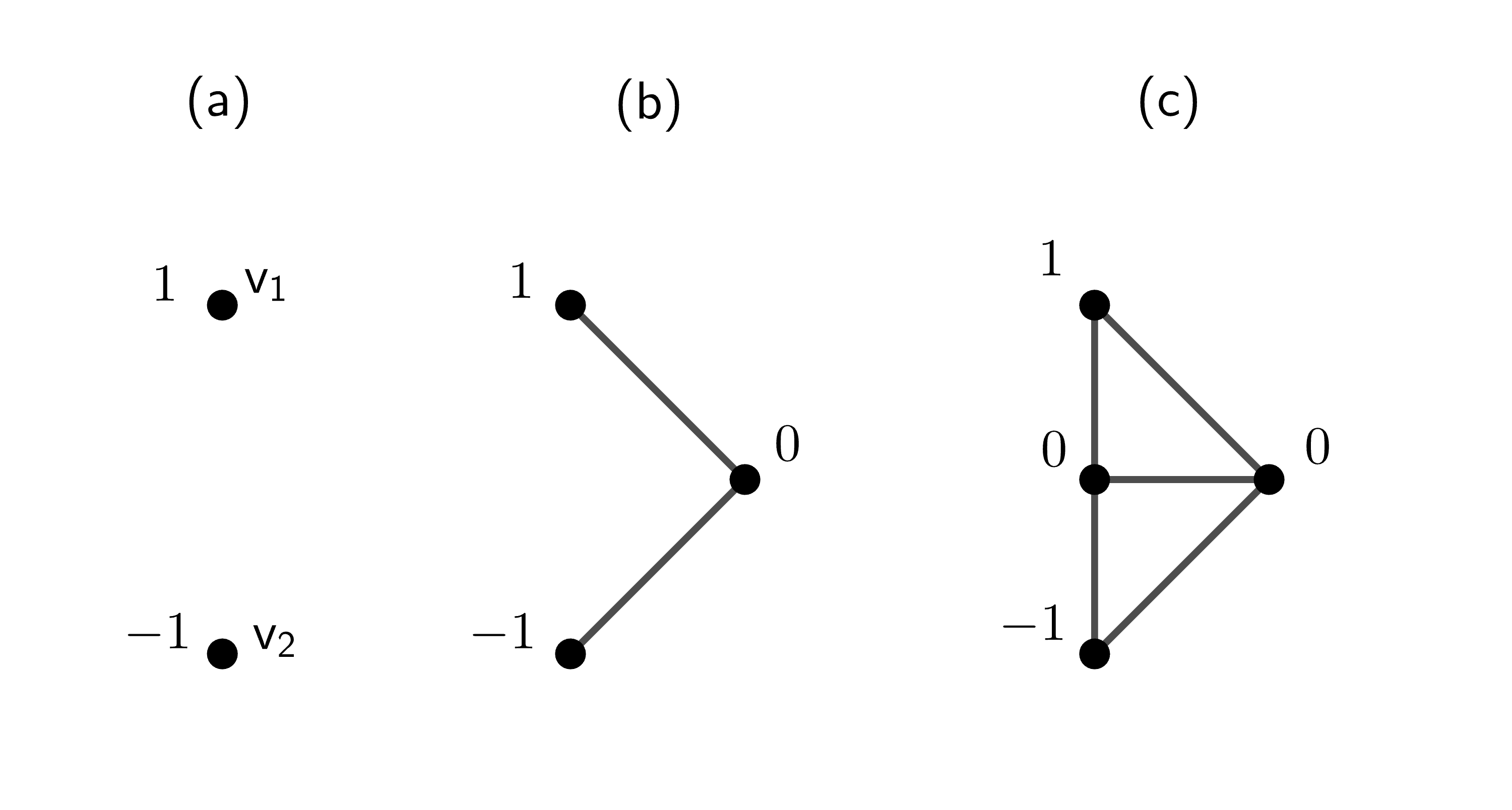}
    \caption{Applying successively the transformation of Theorem \ref{addA} with $W=\{v_1, v_2\}$ to show the evolution of eigenvectors of matrix $A$ associated with the same eigenvalue $\lambda=0$.   }
    \label{fig:graphoperation3}
\end{figure}

\begin{figure}[h]
    \centering
\includegraphics[width=0.5\linewidth]{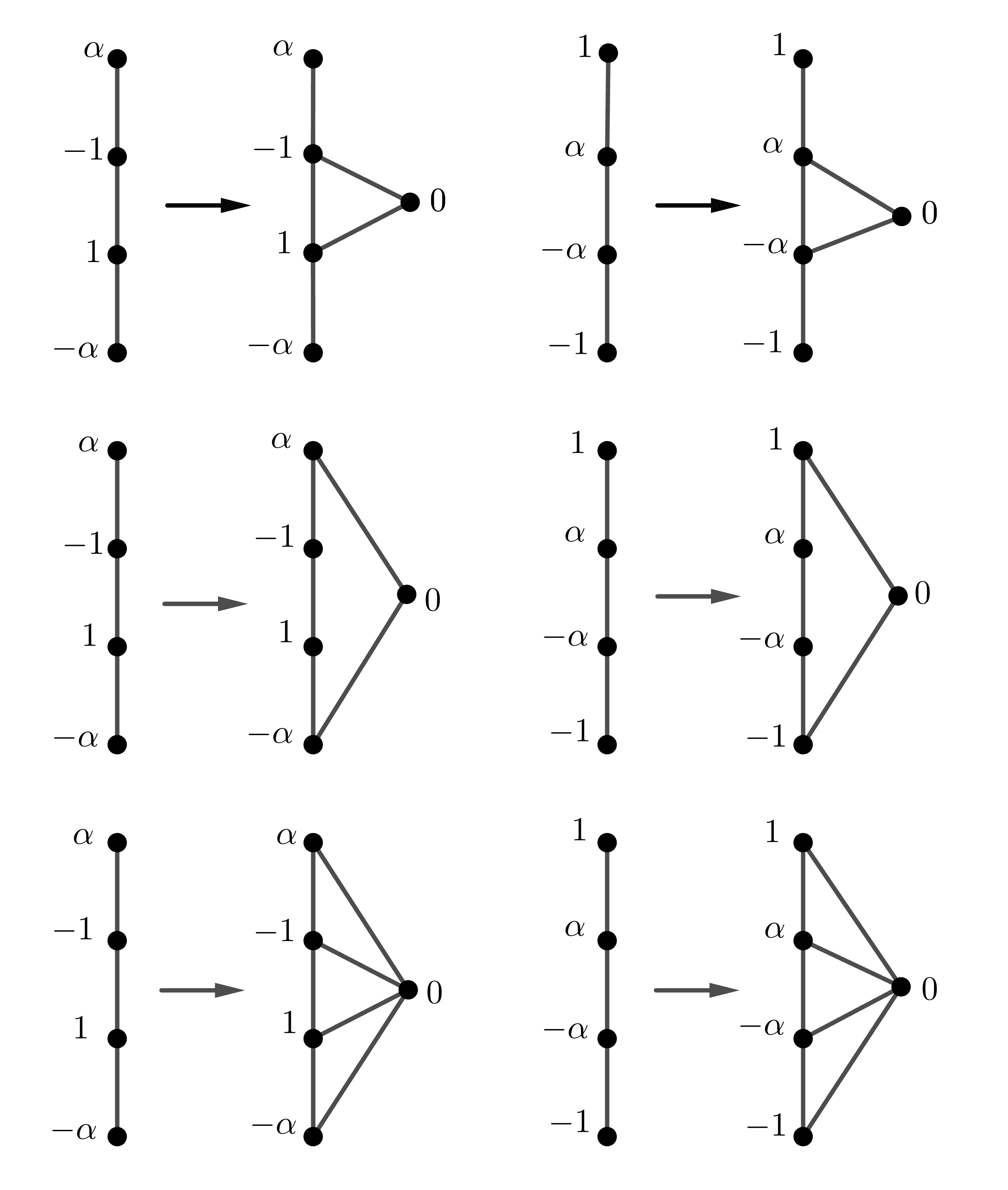}
    \caption{Application of Theorem \ref{addA} for matrix $A$ on $P_4$ on different subsets $W$ with { $\lambda = \frac{\sqrt{5}-1}{2}$ and $\lambda = -\frac{1+ \sqrt{5}}{2}$ (left) and $\lambda = \frac{\sqrt{5}-1}{2}$ (right). }}
    \label{fig:graphoperation4}
\end{figure}

\newpage

\subsection{Articulation}

An important geometric transformation which preserves eigenvalues of $L$ is to ``padd'' the eigenvectors with zeros.  This extends to $Q$ and $A$, and for the three operators $L, Q, A$ we have the following result.
\begin{theorem} Assume a graph $G=(V,E)$ with $n$ vertices where
$\mathbf{x}$ is an eigenvector of L, Q or A such that $x_i=0$ for an eigenvalue $\lambda$
and a vertex $i$. The graph $G'$ where we add a vertex $k$ connected to $i$
has the eigenvector $(\mathbf{x},0)^T$ for the eigenvalue $\lambda$.
\end{theorem}
\begin{proof}
For $Q$
$$d_i x_i + \sum_{j \sim i} x_j = \lambda x_i.$$
Assume $x_i=0$ and adding vertex $k$ such that $x_k=0$ we get
$$(d_i+1)\cdot  0  + \sum_{j \sim i} x_j +0 = 0 = \lambda x_i.$$
{ For vertex $k$, we have 
$$x_{k} + x_{i} = 0 = \lambda x_{k}.$$
}

The proof is trivially extended to $A$.
\end{proof}



{ From now, we write $\lambda_A$, $\lambda_L$ and $\lambda_Q$ to refer to the  eigenvalues of the matrices $A, L$ and $Q$, respectively.} 

Regular graphs are such that $D= dI$ so that
$A,L$ and $Q$ share the same eigenvectors. The eigenvalues are related in the simple way:
$$\lambda_L=-\lambda_A+d, ~~\lambda_Q=\lambda_A+d.$$
{ For $\lambda=0$, we have $L \mathbf{1} = 0$ so that $A \mathbf{1} =d \mathbf{1}$
and $Q\mathbf{1}= 2d \mathbf{1}$. Since $A,Q$ are symetric and real matrices,
an eigenvector $\mathbf{x}$ for a different eigenvalue is orthogonal to 
$\mathbf{1}$ so that $\sum_{i=1}^n x_i=0$. Then $Q$ and $A$ satisfy the conditions of Theorems \ref{addQ} for $Q$ and \ref{addA} for $A$ with $W=V$. }

Reciprocally, the following theorem characterizes graphs such that $A,Q$ and $L$ share a common eigenvector.
\begin{theorem} \label{xaql}
If $\mathbf{x}$ is an eigenvector of $A, L$ and $Q$ affording eigenvalues $\lambda_A,
\lambda_L, \lambda_Q$, then the graph G is such that { every vertex corresponding to a
non zero component of $\mathbf{x}$ has degree $ d=\lambda_Q-\lambda_A = \lambda_L+\lambda_A = (\lambda_L+\lambda_Q)/2$.}
\end{theorem}
\begin{proof}
Take an eigenvector $\textbf{x}$ of $A,$ $L$ and $Q.$ For each vertex $i$, we have 
\begin{eqnarray}
d_{i} x_{i} + \sum_{j \sim i} x_{j} &=& \lambda_{Q} \, x_{i} \label{eq1} \\
d_{i} x_{i} - \sum_{j \sim i} x_{j} &=& \lambda_{L} \, x_{i}. \label{eq2} 
\end{eqnarray}
From \eqref{eq1} - \eqref{eq2}, we get 
\begin{eqnarray}
    \lambda_A = \frac{\lambda_Q -\lambda_{L}}{2}, \label{eq3} \\
2\, d_i \, x_i= (\lambda_Q + \lambda_{L})x_i . ~\label{eq4}
\end{eqnarray}

Equation \eqref{eq4} is satisfied by $x_{i}=0$ or $2d_i = (\lambda_Q + \lambda_{L})$. 
We cannot have $x_i=0$ for all $i$ because $\mathbf{x}$ is a nonzero vector. 
\end{proof}
Following \cite{dam18}, we name {\it soft regular} for a matrix $M$ and an eigenvector $\mathbf{x}$ of $M$ such graphs with the same degree for all nonzero vertices of $\mathbf{x}$.
\begin{remark}
Note that if two of the matrices $A, Q$ and $L$ share an eigenvector, then all of them share the same eigenvector.   
\end{remark}

Figure \ref{fig:trival} displays three graphs having a common eigenvector for matrices $A, Q$ and $L$. 

\begin{figure}[h]
    \centering
\includegraphics[width=0.8\linewidth]{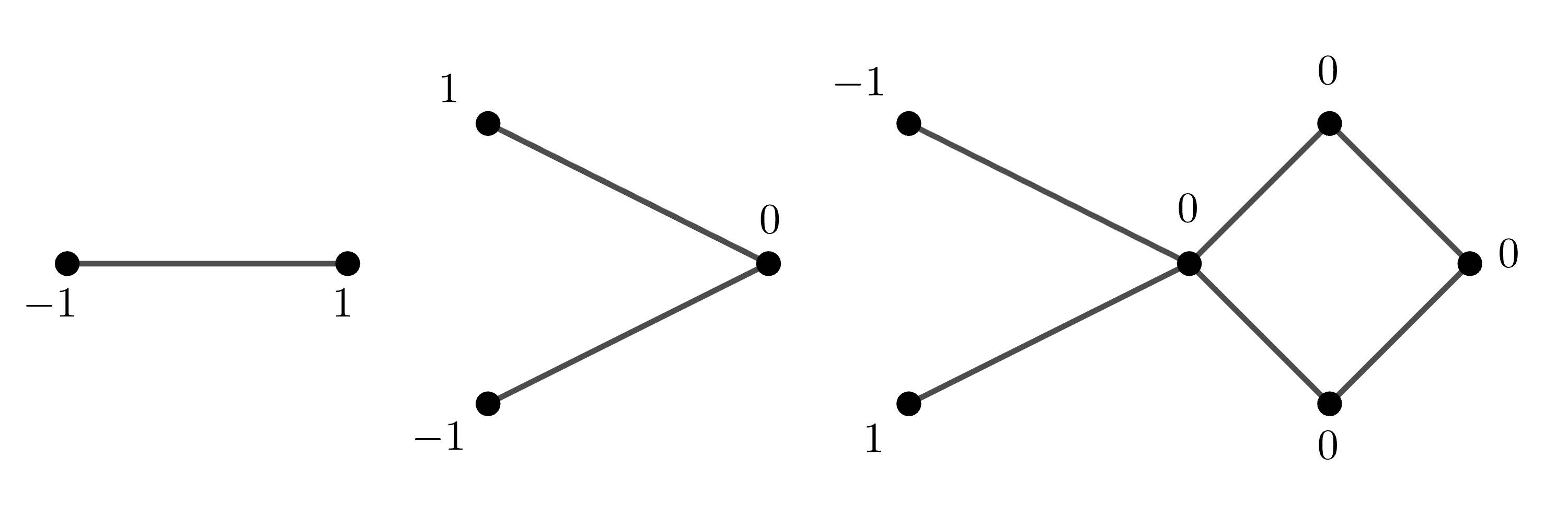}
    \caption{Three graphs having the property of Theorem \ref{xaql} }
    \label{fig:trival}
\end{figure}

{

\section{Bivalent graphs for $L$, $Q$ and $A$}
}
We recall the notion of bivalent graph.
\begin{definition}
A graph is said to be M-bivalent for the operator M, which can be $A, Q$ or $L$, if it admits an 
eigenvector whose components { take the values $\{-1,1\}$} for a given eigenvalue.
\end{definition}

{

$L$-bivalent graphs have been studied in \cite{dam18} and we recall the result below. We then extend the study for matrices $Q$ and 
$A$, and finish with the case of a common eigenvector of $L$, $Q$ and $A$.
}
\subsection{L-bivalent graphs}
{

From \cite{AL21} and \cite{dam18}, we recall the following general characterization of $L$-bivalent graphs.

\begin{theorem}
\label{thm:JA} 
The $L$-bivalent graphs are the regular bipartite graphs and their extensions obtained by adding edges
between nodes having the same value for a bivalent eigenvector.   
\end{theorem}

This result is obtained by using Merris edge principle, deleting all edges between vertices with the same value 
in the considered eigenvector, then using the eigenvalue equation. The graphs can then be classified according to the eigenvalue $\lambda = 2d$, with $d$ the degree of the vertices in the regular bipartite graph.
In the case $\lambda = 0$, the graph $G$ is obtained by partitioning an even number of isolated vertices into two sets, $S$ and $\bar{S}$, of equal cardinality (the sets of vertices assigned $1$ and $-1$, respectively), and then adding any number of edges within either set.
In the case $\lambda = 2 $ ($d=1)$, $G$ is formed from a matching between $S$ and $\bar{S}$  by adding any number of edges in each of the two sets.
The case $\lambda = 4 $ ($d=2$) corresponds to even cycles alternating vertices of $S$ and $\bar{S}$, with any number of additional edges inside $S$ or $\bar{S}$.
For $\lambda = 6 $ ($d=3)$, we take cubic bipartite graphs and add edges again following the edge principle.

}

\subsection{Q-bivalent graphs}

In \cite{CRS07}, the authors prove that every bipartite graph is $Q$-bivalent related to $\lambda_{Q}=0.$
\begin{lemma}(\cite{CRS07})  \label{lem:Qbip}
The least eigenvalue of the signless Laplacian of a connected graph is equal to
0 if and only if the graph is bipartite. In this case 0 is a simple eigenvalue.
\end{lemma}

Following the same approach as in the previous section, we can classify all $Q$-bivalent graphs according to the eigenvalue $\lambda_Q$. To do this, we use the edge principle for $Q$, removing all edges between vertices that have opposite values in the considered eigenvector. Then the 
eigenvalue equation for vertex $i$ reads
$$d_i x_i + \sum_{j\sim i} x_j= \lambda_Q x_i,$$
so that
$$ 2 d_i x_i =\lambda_Q x_i .$$
This shows that the reduced graph has to be $d$-regular and we have $\lambda_Q = 2d$, similarly to the case of $L$.
We then have the following.
\begin{theorem}\label{thmQ}
The $Q$-bivalent graphs are the regular non connected graphs and their extensions obtained by adding any number of edges between some of the connected components and the others.   
\end{theorem}

Note that in the theorem above, some of the initial connected components hold value $1$ and the other $-1$.
Looking at the first possible values for $\lambda_Q$, we find similar results, except that the additional edges are between opposite value vertices instead of equal value vertices, and the two sets $S$ and $\bar{S}$ need not be of same cardinality. More precisely, for $\lambda_Q = 0 $, a $Q$-bivalent graph is obtained from isolated vertices parted in two sets $S$ and $\bar{S}$ and adding any number of edges between the two sets.
In the case $\lambda_Q = 2 $ ($d=1$), we have to consider a matching in $S$ and a matching in $\bar{S}$ and then add edges between the two sets. For $\lambda_Q = 4 $ ($d=2$) we start with cycles in each of the two sets. For $\lambda_Q = 6$ ($d=3$), then we start with cubic graphs in each of the two sets, and so on.

\subsection{A-bivalent graphs}

{
The $A$-bivalent graphs are more difficult to describe. However we have the following general result from (\cite{AL21}) .
}

For a given $S$ be a subset of the vertex set $V$ of a graph ($S \subset V$), let the complement set $\overline{S} = V \setminus S.$ We write $d_{i}(S)$ and  $d_{i}(S,\overline{S})$ for the degree of vertex $i$ inside $S$ and the degree of vertex $i$ with endpoints in  $S$ and $\overline{S}$, respectively.

\begin{theorem}\label{thmAL21}(\cite{AL21})
Let $G=(V,E)$ be a graph on $n$ vertices, and let $S \subset V.$ Then, there is a $\{-1,1\}$-eigenvector associated with eigenvalue $\lambda_A$ for $A$ if and only if 
\begin{equation}\label{eq:bivAthm1}
d_{i}(S)-d_{i}(S, \overline{S}) =\lambda_A \mbox{ for all { vertices} } i \in S,
\end{equation}
and 
\begin{equation}\label{eq:bivAthm2}
d_{i}(\overline{S})-d_{i}(\overline{S}, S) =\lambda_A \;\; \mbox{  for all { vertices } } i \in \overline{S}.
\end{equation}

\end{theorem}

From Theorem \ref{thmAL21}, an $A$-bivalent eigenvector is obtained by assigning $\mathbf{x}_{i}=1$ if $i \in S$ and $\mathbf{x}_{i}=-1$ if $i \in \overline{S}$.

\subsection{Common bivalent eigenvector for $L$, $Q$ and $A$}

The next result is an application of Theorem \ref{xaql} when the eigenvector $\mathbf{x}$ is bivalent, that is, $\mathbf{x}$ has no zero entries. In this case, the nonzero entries of $\mathbf{x}$ induce a regular graph. 
\begin{theorem}\label{thm35}
Assume $\mathbf{x}$ is a common bivalent eigenvector for $A, L$ and $Q$. Then the graph is regular.
It is formed from a regular bipartite graph with any number of additional edges within the two sets of the bipartite graph such that the additional edges induce a regular graph.
\end{theorem}


{ The reciprocal of Theorem 3.5 is not true, we should require $G$ to be bipartite.} 
\begin{theorem}
Consider a regular bipartite graph $G$. Then
it has a common bivalent eigenvector $\mathbf{x}$ for all three matrices $A, L$ and $Q$. 
\end{theorem}

 Recall that bipartite graphs have symmetric spectrum for the adjacency matrix, \cite{CRS97}. If in addition $G$ is $d$-regular then $\pm d$ are eigenvalues of $A$.  The constant eigenvector 
$[1,\ldots,1,\ldots,1]$ is an eigenvector for $d$ for $A,L$ and $Q$.  The eigenvector for $-d$ and $A$, $[1,\ldots,1,-1,\ldots,-1]$ is bivalent. Since $G$ is regular, $A, L$ and $Q$ share this same bivalent eigenvector. The first graph of Figure \ref{fig:abivalentregular} is an example of such a bivalent eigenvector for a bipartite regular graph. The graphs $G_1$ and $G_2$ are illustrations of Theorem \ref{thm35}.


\begin{figure}[H]
    \centering
    \includegraphics[width=0.5\linewidth]{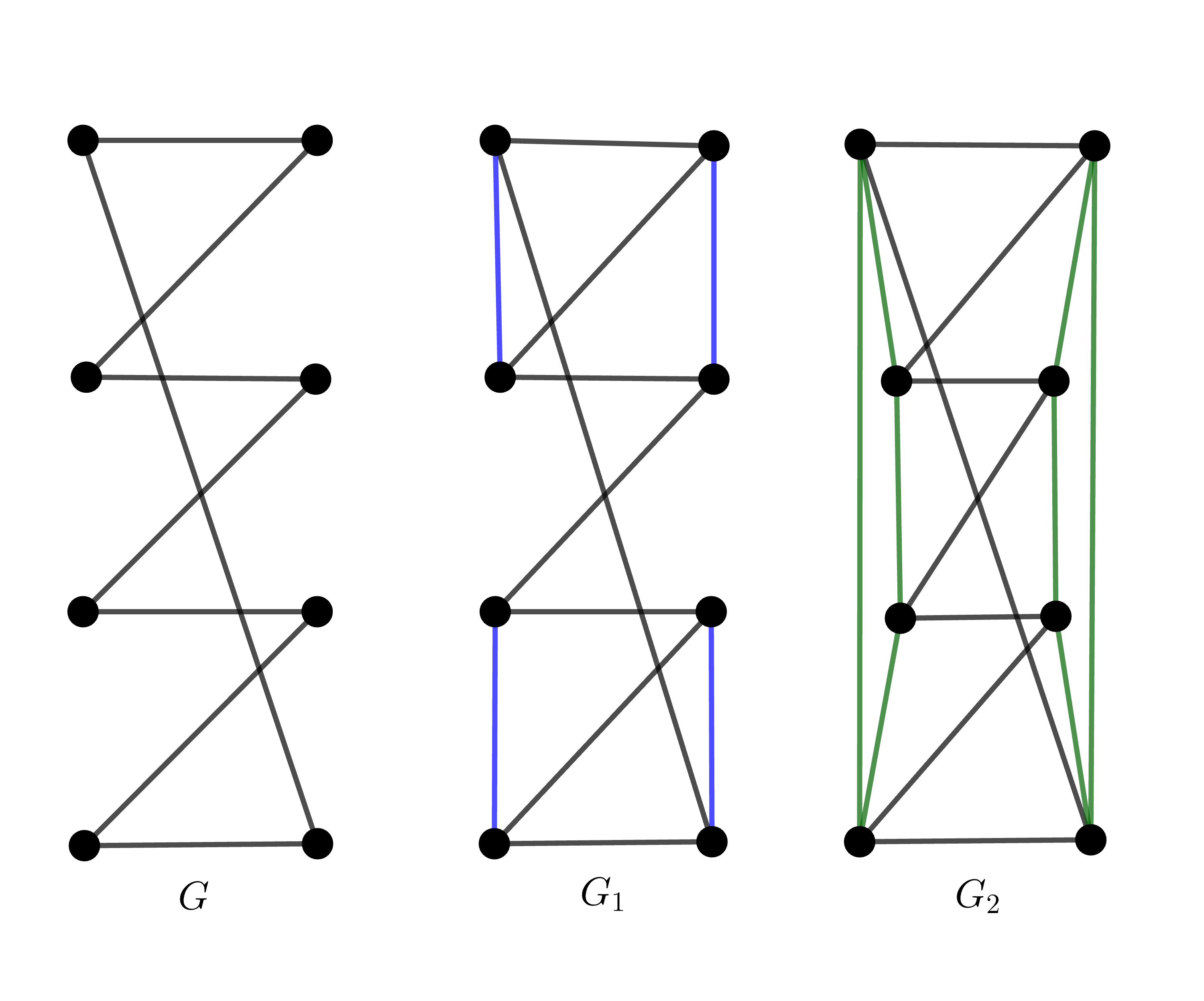}
    \caption{Graph $G$ is $C_8$ and $A$-bivalent associated with $\lambda_A=-2$; $G_1$ is $A$-bivalent associated with $\lambda_A=-1$; $G_2$ is $A$-bivalent associated with $\lambda_A = 0.$ 
    }
    \label{fig:abivalentregular}
\end{figure}

\subsection{Example : circulant graphs}

We call circulant a graph when its adjacency matrix $A$ is circulant and observe that in this case the matrices $Q$ and $L$ are circulant as well.
It is well-known \cite{bc2009} that the eigenvectors of any circulant matrix are given, for $0 \leq k \leq n-1$,  by $$v_{k} = (1, w^{k}, w^{2k},\ldots, w^{(n-1)k}),$$ where $w^{k} = \cos(\frac{2\pi k}{n}) + i \sin(\frac{2\pi k}{n})$  .

For circulant graphs, we have the following result.
\begin{proposition}\label{prop:1}
Let $G$ be a circulant graph with $n$ even vertices. Then, $G$ is bivalent for $A,Q$ and $L$. 
\end{proposition}
\begin{proof}
Based on the above formula, 
it is easy to see that $(1,-1,1,-1,\dots)^T$ ($n$ multiple of 2) and $(1,-1,-1,1,\dots)^T$ ($n$ multiple of 4) are eigenvectors.
\end{proof}

An immediate consequence of Proposition \ref{prop:1} is that
only even cycles (2-regular bipartite) are bivalent for $A, Q$ and $L$. 
Specifically, for $n=2k$, $\lambda_{A}=-2,~\lambda_{Q}=0$
and for $n=4k$, $\lambda_{A}=0,~\lambda_{Q}=2$.




\section{Bivalent trees}

We study in this section the case of bivalent trees, starting with the particular case of paths.

\subsection{Paths}

{
$L$-bivalent paths are the ones with even order, see \cite{AJAL26+} and \cite{dam18}.
}

{  We prove that there are no $A$-bivalent path graphs of order greater than or equal to $3$.} Note that $P_2$ is $A$-bivalent since eigenvalues $\lambda_A=1$ and $\lambda_A = -1$ are associated with bivalent eigenvectors.

\begin{proposition}
Let $P_n$ be a path on $n \geq 3$ vertices. Then, $P_n$ is not $A$-bivalent.
\end{proposition}
\begin{proof}
 Write the vertex set of $P_n$ as $V = S \cup \overline{S}$ such that $S \cap \overline{S} = \emptyset.$  Let $u \in V$ be a leaf, $v \sim u$ and $w \sim v.$ \\
\noindent Consider the following cases: 

\noindent \textit{Case 1:} Let $u \in S$ and $v \in \overline{S}:$ note that 
$$d_u(S) - d_u(S,\overline{S}) = \lambda_A = -1.$$
Now, let $w \sim v$. If $w \in \overline{S},$ then $d_v(\overline{S}) - d_v(\overline{S}, S) = \lambda_A = 0,$ which is a contradiction. If $w \in S,$ we obtain $d_v(\overline{S}) - d_v(\overline{S}, S) = \lambda_A = -2,$ which is a contradiction.

\noindent  \textit{Case 2:} Let $u, v \in S:$ in this case, 
$$d_u(S) - d_u(S,\overline{S}) = \lambda_A = 1.$$ If $w \in \overline{S},$ then $d_v(S) - d_v(S, \overline{S}) = \lambda_A = 0,$ which is a contradiction. If $w \in S,$ then $d_v(S) - d_v(S, \overline{S}) = \lambda_A = 2,$ which is a contradiction. 
Therefore, there are no $A$-bivalent paths for $n\geq 3$.
\end{proof}

For $Q$ we have the following.
\begin{proposition}
Let $P_n$ be a path on $n \geq 4$ vertices. Then, $P_n$ is $Q$-bivalent associated with $\lambda_{Q} = 0$ for any $n$ and $Q$-bivalent associated with $\lambda_{Q} = 2$ when $n$ is even.
\end{proposition}
\begin{proof}
It is clear from Lemma \ref{lem:Qbip}, that $\lambda_{Q} = 0$ is associated with a $Q$-bivalent eigenvector, and the first case is proved for any $n$ { $\geq 4 $}.  

Now, let $P_{n}: e_1\, e_2\,\ldots e_{n-1}$ where $e_{i}=(v_{i}, v_{i+1})$ is an edge connecting vertices $v_i$ and $v_{i+1}.$ There are two ways of arranging the  vertices of $P_n$ in the sets $S$ and $\overline{S}$. 
First, consider that we arrange the edges of $P_n$ such that there are no edges between the vertices of $S$ and no edges between the vertices of $\overline{S}$. This is exactly the same case as the previous one, which was already proved. Second, we can arrange the edges such that $e_{2k+1} \in S$ and $e_{2k+3} \in \overline{S}.$ If $n$ is even, then the graphs induced by $S$ and $\overline{S}$ are unions of paths $P_2.$ From Theorem \ref{thmAL21}, $P_n$ is $Q$-bivalent associated with $\lambda_{Q} = 2.$ If $n$ is odd, then one of the two graphs induced by $S$ or $\overline{S}$  is a union of paths $P_2$ and the other one is a union of paths $P_2$ and an isolated vertex. From Theorem \ref{thmAL21}, this is a contradiction, and the result follows.
\end{proof}

\subsection{Trees \label{subsec:44}}

{
In a previous work, \cite{dam18}  we characterize $L$-bivalent trees as the ones having a perfect matching. Equivalently in \cite{AJAL26+}, we characterize bivalent trees as unions of $P_2$ connected with edges between vertices corresponding to equal entries in the considered eigenvector. 
}

Since a tree is bipartite, it is $Q$-bivalent for eigenvalue $0,$ [Proposition 2.1 of \cite{CRS07}]. 
{
Following the result in Theorem \ref{thmQ}, a $Q$-bivalent tree can also be obtained for $\lambda=2$ with a matching on $n$ vertices, connecting the components with exactly $n-1$ edges without creating a cycle. There are no other possibilities.
}

From Theorem \ref{thmAL21}, the bivalent eigenvectors of an $A$-bivalent tree are associated with the very restricted set of eigenvalues $\lambda_{A} \in \{-1,1\}.$

\begin{proposition}\label{prop310}
Let $T$ be a tree on $n$ vertices. If $T$ is $A$-bivalent, the bivalent eigenvector is associated with eigenvalues $\lambda_A=1$ or $\lambda_A=-1.$
\end{proposition}
\begin{proof}
Let $v$ be a leaf of $T$ and $u$ its neighbor. If $u$ and $v$ are in the same set, say $S$, then $d_{v}(S) - d_{v}(S,\overline{S}) = \lambda_A = 1.$ Since $T$ is $A$-bivalent, from Theorem \ref{thmAL21}, we should arrange all vertices of $T$ such that this same equation is true for $\lambda_A=1.$ Then, $\lambda_A = 1.$ 
If $v$ and $u$ are in different sets, then 
$d_{v}(S) - d_{v}(S,\overline{S}) = \lambda_A = -1.$ Using the same previous argument, this equation should be true for all vertices of the tree, that is, $\lambda_A=-1,$ and the result follows.
\end{proof}

\begin{lemma}\label{lemma:1}
Let $T \;\; (\ne P_n$) be an $A$-bivalent tree on  $n$ vertices. The following conditions hold:
\begin{itemize}
    \item [(a) ] All leaves, incident to the same vertex, should be in the same set.
    \item [(b) ] There are no two adjacent vertices, say $u$ and $v$ of degrees 2 and 1, respectively, in the same set.
    \item [(c) ] There is no vertex $w \in S$ such that $d_{w}(S) \geq 2$ and $d_{w}(S,\overline{S}) = 0.$ The same holds if $w \in \overline{S}.$
\end{itemize}
\end{lemma}
\begin{proof}
\noindent Write the vertex set of $T$ as $V = S \cup \overline{S}$ such that $S \cap \overline{S} = \emptyset.$ \\ 
\noindent (a)  Suppose that $u \in S$ and $v \in \overline{S} $ are two leaves of $T$ incident to the same vertex. Let $w \sim u$. In this case,
$d_u(S) - d_u(S, \overline{S}) = \lambda_A = -1.$ On the other hand, 
$d_v(\overline{S}) - d_u( \overline{S}, S) = \lambda_A = 1,$ which is a contradiction.

\noindent (b) Let $uv$ be an edge in $S$ such that $d_{u}(T) = 2$ and $d_{v}(T) = 1,$ and let $w \in N(u).$ Note that  
$d_u(S) - d_u(S, \overline{S}) = \lambda_A \in \{0,2\}$ depending whether $w$ is in $\overline{S}$ or $S$. On the other hand, 
$d_v(S) - d_v(S, \overline{S}) = \lambda_A = 1,$ which is a contradiction.

\noindent (c) for a leaf $v \in S, \;   d_{v}(S) - d_{v}(S,\overline{S}) = \lambda_A \in \{-1,1\}.$ For $w \in S, \; d_{w}(S) - d_{w}(S,\overline{S}) = \lambda_A \geq 2,$ which is a contradiction. 
\end{proof}

Next, we prove that $A$-bivalent trees are trees with only odd degrees.
\begin{theorem} \label{thm:odddegree}
    If a tree $T$ on $n$ vertices is $A$-bivalent then all of its vertices are of odd degree. Moreover, $n$ is even.
\end{theorem}
\begin {proof}
 Consider a tree $T$ with an $A$-eigenvector $\mathbf{x}$ affording an eigenvalue $\lambda_A$. We already know from Proposition \ref{prop310} that  $\lambda_A =1$ or $\lambda_A =-1$. 
Then the eigenvalue equation reads : 
$$
\sum_{j \sim i} x_j = \lambda_A x_i \qquad \forall i \in V.
$$
This implies that any vertex will have an odd number of adjacent vertices, with a difference of 1 in the count of $1's$ and -$1's.$ 
Also, since the sum of the vertex degrees of a graph is even and each vertex of the tree has odd degree, then $n$ should be even. 
\end{proof}
 More precisely in the proof above, when $\lambda_A = 1$, there is one more neighbor having the same value as $x_i$ when compared to the number of neighbors holding the opposite value. and when $\lambda_A=-1$, the number of neighbors holding the same value than $x_i$ is one unit less than the number of neighbors holding the opposite value.
In the particular case where vertex $i$ is a leaf, the above equation implies, when the eigenvalue is $1$, that the only vertex $j$ adjacent to $i$ holds the same value: $x_i=x_j$. And when $\lambda_A =-1$, the leaf and its incident vertex need to have opposite values. 

\begin{remark}
The reciprocal of Theorem \ref{thm:odddegree} is not true. See for instance the graph $K_{1,2k+1}$ for any integer $k$, which is not $A$-bivalent. 
\end{remark}


We conclude by observing that if $T$ is an $A-$bivalent tree, it should have an adequate number of pendant vertices attached to each non-pendant vertex. Figure \ref{fig:Abivtree} displays an example of an $A$-bivalent tree satisfying all properties of subsection \ref{subsec:44}.

\begin{figure}[H]
\centering    
\begin{tikzpicture}[scale=0.7, transform shape]
\definecolor{cv0}{rgb}{0.0,0.0,0.0}
\definecolor{cfv0}{rgb}{1.0,1.0,1.0}
\definecolor{clv0}{rgb}{0.0,0.0,0.0}
\definecolor{cv1}{rgb}{0.0,0.0,0.0}
\definecolor{cfv1}{rgb}{1.0,1.0,1.0}
\definecolor{clv1}{rgb}{0.0,0.0,0.0}
\definecolor{cv2}{rgb}{0.0,0.0,0.0}
\definecolor{cfv2}{rgb}{1.0,1.0,1.0}
\definecolor{clv2}{rgb}{0.0,0.0,0.0}
\definecolor{cv3}{rgb}{0.0,0.0,0.0}
\definecolor{cfv3}{rgb}{1.0,1.0,1.0}
\definecolor{clv3}{rgb}{0.0,0.0,0.0}
\definecolor{cv4}{rgb}{0.0,0.0,0.0}
\definecolor{cfv4}{rgb}{1.0,1.0,1.0}
\definecolor{clv4}{rgb}{0.0,0.0,0.0}
\definecolor{cv5}{rgb}{0.0,0.0,0.0}
\definecolor{cfv5}{rgb}{1.0,1.0,1.0}
\definecolor{clv5}{rgb}{0.0,0.0,0.0}
\definecolor{cv6}{rgb}{0.0,0.0,0.0}
\definecolor{cfv6}{rgb}{1.0,1.0,1.0}
\definecolor{clv6}{rgb}{0.0,0.0,0.0}
\definecolor{cv7}{rgb}{0.0,0.0,0.0}
\definecolor{cfv7}{rgb}{1.0,1.0,1.0}
\definecolor{clv7}{rgb}{0.0,0.0,0.0}
\definecolor{cv8}{rgb}{0.0,0.0,0.0}
\definecolor{cfv8}{rgb}{1.0,1.0,1.0}
\definecolor{clv8}{rgb}{0.0,0.0,0.0}
\definecolor{cv9}{rgb}{0.0,0.0,0.0}
\definecolor{cfv9}{rgb}{1.0,1.0,1.0}
\definecolor{clv9}{rgb}{0.0,0.0,0.0}
\definecolor{cv10}{rgb}{0.0,0.0,0.0}
\definecolor{cfv10}{rgb}{1.0,1.0,1.0}
\definecolor{clv10}{rgb}{0.0,0.0,0.0}
\definecolor{cv11}{rgb}{0.0,0.0,0.0}
\definecolor{cfv11}{rgb}{1.0,1.0,1.0}
\definecolor{clv11}{rgb}{0.0,0.0,0.0}
\definecolor{cv12}{rgb}{0.0,0.0,0.0}
\definecolor{cfv12}{rgb}{1.0,1.0,1.0}
\definecolor{clv12}{rgb}{0.0,0.0,0.0}
\definecolor{cv13}{rgb}{0.0,0.0,0.0}
\definecolor{cfv13}{rgb}{1.0,1.0,1.0}
\definecolor{clv13}{rgb}{0.0,0.0,0.0}
\definecolor{cv14}{rgb}{0.0,0.0,0.0}
\definecolor{cfv14}{rgb}{1.0,1.0,1.0}
\definecolor{clv14}{rgb}{0.0,0.0,0.0}
\definecolor{cv15}{rgb}{0.0,0.0,0.0}
\definecolor{cfv15}{rgb}{1.0,1.0,1.0}
\definecolor{clv15}{rgb}{0.0,0.0,0.0}
\definecolor{cv16}{rgb}{0.0,0.0,0.0}
\definecolor{cfv16}{rgb}{1.0,1.0,1.0}
\definecolor{clv16}{rgb}{0.0,0.0,0.0}
\definecolor{cv17}{rgb}{0.0,0.0,0.0}
\definecolor{cfv17}{rgb}{1.0,1.0,1.0}
\definecolor{clv17}{rgb}{0.0,0.0,0.0}
\definecolor{cv18}{rgb}{0.0,0.0,0.0}
\definecolor{cfv18}{rgb}{1.0,1.0,1.0}
\definecolor{clv18}{rgb}{0.0,0.0,0.0}
\definecolor{cv19}{rgb}{0.0,0.0,0.0}
\definecolor{cfv19}{rgb}{1.0,1.0,1.0}
\definecolor{clv19}{rgb}{0.0,0.0,0.0}
\definecolor{cv20}{rgb}{0.0,0.0,0.0}
\definecolor{cfv20}{rgb}{1.0,1.0,1.0}
\definecolor{clv20}{rgb}{0.0,0.0,0.0}
\definecolor{cv21}{rgb}{0.0,0.0,0.0}
\definecolor{cfv21}{rgb}{1.0,1.0,1.0}
\definecolor{clv21}{rgb}{0.0,0.0,0.0}
\definecolor{cv8v9}{rgb}{0.0,0.0,0.0}
\definecolor{cv8v10}{rgb}{0.0,0.0,0.0}
\definecolor{cv12v13}{rgb}{0.0,0.0,0.0}
\definecolor{cv13v14}{rgb}{0.0,0.0,0.0}
\definecolor{cv13v19}{rgb}{0.0,0.0,0.0}
\definecolor{cv14v15}{rgb}{0.0,0.0,0.0}
\definecolor{cv14v20}{rgb}{0.0,0.0,0.0}
\definecolor{cv14v21}{rgb}{0.0,0.0,0.0}
\definecolor{cv14v0}{rgb}{0.0,0.0,0.0}
\definecolor{cv15v16}{rgb}{0.0,0.0,0.0}
\definecolor{cv15v1}{rgb}{0.0,0.0,0.0}
\definecolor{cv15v2}{rgb}{0.0,0.0,0.0}
\definecolor{cv15v3}{rgb}{0.0,0.0,0.0}
\definecolor{cv16v6}{rgb}{0.0,0.0,0.0}
\definecolor{cv16v7}{rgb}{0.0,0.0,0.0}
\definecolor{cv16v8}{rgb}{0.0,0.0,0.0}
\definecolor{cv16v17}{rgb}{0.0,0.0,0.0}
\definecolor{cv16v4}{rgb}{0.0,0.0,0.0}
\definecolor{cv16v5}{rgb}{0.0,0.0,0.0}
\definecolor{cv17v11}{rgb}{0.0,0.0,0.0}
\definecolor{cv17v18}{rgb}{0.0,0.0,0.0}
\Vertex[style={minimum size=1.0cm,draw=cv0,fill=cfv0,text=clv0,shape=circle},LabelOut=false,L=\hbox{$-1$},x=7.87cm,y=2.85cm]{v0}
\Vertex[style={minimum size=1.0cm,draw=cv1,fill=cfv1,text=clv1,shape=circle},LabelOut=false,L=\hbox{$-1$},x=6.04cm,y=4.16cm]{v1}
\Vertex[style={minimum size=1.0cm,draw=cv2,fill=cfv2,text=clv2,shape=circle},LabelOut=false,L=\hbox{$-1$},x=4.14cm,y=3.18cm]{v2}
\Vertex[style={minimum size=1.0cm,draw=cv3,fill=cfv3,text=clv3,shape=circle},LabelOut=false,L=\hbox{$-1$},x=5.00cm,y=2.87cm]{v3}
\Vertex[style={minimum size=1.0cm,draw=cv4,fill=cfv4,text=clv4,shape=circle},LabelOut=false,L=\hbox{$-1$},x=3.7cm,y=6.11cm]{v4}
\Vertex[style={minimum size=1.0cm,draw=cv5,fill=cfv5,text=clv5,shape=circle},LabelOut=false,L=\hbox{$-1$},x=4.30cm,y=5.62cm]{v5}
\Vertex[style={minimum size=1.0cm,draw=cv6,fill=cfv6,text=clv6,shape=circle},LabelOut=false,L=\hbox{$-1$},x=2.36cm,y=5.85cm]{v6}
\Vertex[style={minimum size=1.0cm,draw=cv7,fill=cfv7,text=clv7,shape=circle},LabelOut=false,L=\hbox{$-1$},x=2.39cm,y=4.63cm]{v7}
\Vertex[style={minimum size=1.0cm,draw=cv8,fill=cfv8,text=clv8,shape=circle},LabelOut=false,L=\hbox{$1$},x=1.22cm,y=5.44cm]{v8}
\Vertex[style={minimum size=1.0cm,draw=cv9,fill=cfv9,text=clv9,shape=circle},LabelOut=false,L=\hbox{$-1$},x=0.00cm,y=5.11cm]{v9}
\Vertex[style={minimum size=1.0cm,draw=cv10,fill=cfv10,text=clv10,shape=circle},LabelOut=false,L=\hbox{$-1$},x=0.08cm,y=5.98cm]{v10}
\Vertex[style={minimum size=1.0cm,draw=cv11,fill=cfv11,text=clv11,shape=circle},LabelOut=false,L=\hbox{$-1$},x=3.80cm,y=8.00cm]{v11}
\Vertex[style={minimum size=1.0cm,draw=cv12,fill=cfv12,text=clv12,shape=circle},LabelOut=false,L=\hbox{$-1$},x=8.14cm,y=0.00cm]{v12}
\Vertex[style={minimum size=1.0cm,draw=cv13,fill=cfv13,text=clv13,shape=circle},LabelOut=false,L=\hbox{$1$},x=7.88cm,y=1.05cm]{v13}
\Vertex[style={minimum size=1.0cm,draw=cv14,fill=cfv14,text=clv14,shape=circle},LabelOut=false,L=\hbox{$1$},x=6.78cm,y=2.39cm]{v14}
\Vertex[style={minimum size=1.0cm,draw=cv15,fill=cfv15,text=clv15,shape=circle},LabelOut=false,L=\hbox{$1$},x=4.99cm,y=3.75cm]{v15}
\Vertex[style={minimum size=1.0cm,draw=cv16,fill=cfv16,text=clv16,shape=circle},LabelOut=false,L=\hbox{$1$},x=3.26cm,y=5.30cm]{v16}
\Vertex[style={minimum size=1.0cm,draw=cv17,fill=cfv17,text=clv17,shape=circle},LabelOut=false,L=\hbox{$1$},x=3.27cm,y=7.01cm]{v17}
\Vertex[style={minimum size=1.0cm,draw=cv18,fill=cfv18,text=clv18,shape=circle},LabelOut=false,L=\hbox{$-1$},x=2.80cm,y=7.98cm]{v18}
\Vertex[style={minimum size=1.0cm,draw=cv19,fill=cfv19,text=clv19,shape=circle},LabelOut=false,L=\hbox{$-1$},x=9.00cm,y=0.38cm]{v19}
\Vertex[style={minimum size=1.0cm,draw=cv20,fill=cfv20,text=clv20,shape=circle},LabelOut=false,L=\hbox{$-1$},x=6.56cm,y=1.49cm]{v20}
\Vertex[style={minimum size=1.0cm,draw=cv21,fill=cfv21,text=clv21,shape=circle},LabelOut=false,L=\hbox{$-1$},x=7.87cm,y=2.14cm]{v21}
\Edge[lw=0.05cm,style={color=cv8v9,},](v8)(v9)
\Edge[lw=0.05cm,style={color=cv8v10,},](v8)(v10)
\Edge[lw=0.05cm,style={color=cv12v13,},](v12)(v13)
\Edge[lw=0.05cm,style={color=cv13v14,},](v13)(v14)
\Edge[lw=0.05cm,style={color=cv13v19,},](v13)(v19)
\Edge[lw=0.05cm,style={color=cv14v15,},](v14)(v15)
\Edge[lw=0.05cm,style={color=cv14v20,},](v14)(v20)
\Edge[lw=0.05cm,style={color=cv14v21,},](v14)(v21)
\Edge[lw=0.05cm,style={color=cv14v0,},](v14)(v0)
\Edge[lw=0.05cm,style={color=cv15v16,},](v15)(v16)
\Edge[lw=0.05cm,style={color=cv15v1,},](v15)(v1)
\Edge[lw=0.05cm,style={color=cv15v2,},](v15)(v2)
\Edge[lw=0.05cm,style={color=cv15v3,},](v15)(v3)
\Edge[lw=0.05cm,style={color=cv16v6,},](v16)(v6)
\Edge[lw=0.05cm,style={color=cv16v7,},](v16)(v7)
\Edge[lw=0.05cm,style={color=cv16v8,},](v16)(v8)
\Edge[lw=0.05cm,style={color=cv16v17,},](v16)(v17)
\Edge[lw=0.05cm,style={color=cv16v4,},](v16)(v4)
\Edge[lw=0.05cm,style={color=cv16v5,},](v16)(v5)
\Edge[lw=0.05cm,style={color=cv17v11,},](v17)(v11)
\Edge[lw=0.05cm,style={color=cv17v18,},](v17)(v18)
\end{tikzpicture}
\caption{An $A$-bivalent tree associated with $\lambda_A=-1$}
\label{fig:Abivtree}
\end{figure}

We observe that the same tree of Figure \ref{fig:Abivtree} is also $A$-bivalent to  $\lambda_A = 1$ associated with a different eigenvector that can be observed by the labels of the vertices in Figure \ref{fig:Abivtree2}.

\begin{figure}[H]
\centering    
\begin{tikzpicture}[scale=0.7, transform shape]
\definecolor{cv0}{rgb}{0.0,0.0,0.0}
\definecolor{cfv0}{rgb}{1.0,1.0,1.0}
\definecolor{clv0}{rgb}{0.0,0.0,0.0}
\definecolor{cv1}{rgb}{0.0,0.0,0.0}
\definecolor{cfv1}{rgb}{1.0,1.0,1.0}
\definecolor{clv1}{rgb}{0.0,0.0,0.0}
\definecolor{cv2}{rgb}{0.0,0.0,0.0}
\definecolor{cfv2}{rgb}{1.0,1.0,1.0}
\definecolor{clv2}{rgb}{0.0,0.0,0.0}
\definecolor{cv3}{rgb}{0.0,0.0,0.0}
\definecolor{cfv3}{rgb}{1.0,1.0,1.0}
\definecolor{clv3}{rgb}{0.0,0.0,0.0}
\definecolor{cv4}{rgb}{0.0,0.0,0.0}
\definecolor{cfv4}{rgb}{1.0,1.0,1.0}
\definecolor{clv4}{rgb}{0.0,0.0,0.0}
\definecolor{cv5}{rgb}{0.0,0.0,0.0}
\definecolor{cfv5}{rgb}{1.0,1.0,1.0}
\definecolor{clv5}{rgb}{0.0,0.0,0.0}
\definecolor{cv6}{rgb}{0.0,0.0,0.0}
\definecolor{cfv6}{rgb}{1.0,1.0,1.0}
\definecolor{clv6}{rgb}{0.0,0.0,0.0}
\definecolor{cv7}{rgb}{0.0,0.0,0.0}
\definecolor{cfv7}{rgb}{1.0,1.0,1.0}
\definecolor{clv7}{rgb}{0.0,0.0,0.0}
\definecolor{cv8}{rgb}{0.0,0.0,0.0}
\definecolor{cfv8}{rgb}{1.0,1.0,1.0}
\definecolor{clv8}{rgb}{0.0,0.0,0.0}
\definecolor{cv9}{rgb}{0.0,0.0,0.0}
\definecolor{cfv9}{rgb}{1.0,1.0,1.0}
\definecolor{clv9}{rgb}{0.0,0.0,0.0}
\definecolor{cv10}{rgb}{0.0,0.0,0.0}
\definecolor{cfv10}{rgb}{1.0,1.0,1.0}
\definecolor{clv10}{rgb}{0.0,0.0,0.0}
\definecolor{cv11}{rgb}{0.0,0.0,0.0}
\definecolor{cfv11}{rgb}{1.0,1.0,1.0}
\definecolor{clv11}{rgb}{0.0,0.0,0.0}
\definecolor{cv12}{rgb}{0.0,0.0,0.0}
\definecolor{cfv12}{rgb}{1.0,1.0,1.0}
\definecolor{clv12}{rgb}{0.0,0.0,0.0}
\definecolor{cv13}{rgb}{0.0,0.0,0.0}
\definecolor{cfv13}{rgb}{1.0,1.0,1.0}
\definecolor{clv13}{rgb}{0.0,0.0,0.0}
\definecolor{cv14}{rgb}{0.0,0.0,0.0}
\definecolor{cfv14}{rgb}{1.0,1.0,1.0}
\definecolor{clv14}{rgb}{0.0,0.0,0.0}
\definecolor{cv15}{rgb}{0.0,0.0,0.0}
\definecolor{cfv15}{rgb}{1.0,1.0,1.0}
\definecolor{clv15}{rgb}{0.0,0.0,0.0}
\definecolor{cv16}{rgb}{0.0,0.0,0.0}
\definecolor{cfv16}{rgb}{1.0,1.0,1.0}
\definecolor{clv16}{rgb}{0.0,0.0,0.0}
\definecolor{cv17}{rgb}{0.0,0.0,0.0}
\definecolor{cfv17}{rgb}{1.0,1.0,1.0}
\definecolor{clv17}{rgb}{0.0,0.0,0.0}
\definecolor{cv18}{rgb}{0.0,0.0,0.0}
\definecolor{cfv18}{rgb}{1.0,1.0,1.0}
\definecolor{clv18}{rgb}{0.0,0.0,0.0}
\definecolor{cv19}{rgb}{0.0,0.0,0.0}
\definecolor{cfv19}{rgb}{1.0,1.0,1.0}
\definecolor{clv19}{rgb}{0.0,0.0,0.0}
\definecolor{cv20}{rgb}{0.0,0.0,0.0}
\definecolor{cfv20}{rgb}{1.0,1.0,1.0}
\definecolor{clv20}{rgb}{0.0,0.0,0.0}
\definecolor{cv21}{rgb}{0.0,0.0,0.0}
\definecolor{cfv21}{rgb}{1.0,1.0,1.0}
\definecolor{clv21}{rgb}{0.0,0.0,0.0}
\definecolor{cv8v9}{rgb}{0.0,0.0,0.0}
\definecolor{cv8v10}{rgb}{0.0,0.0,0.0}
\definecolor{cv12v13}{rgb}{0.0,0.0,0.0}
\definecolor{cv13v14}{rgb}{0.0,0.0,0.0}
\definecolor{cv13v19}{rgb}{0.0,0.0,0.0}
\definecolor{cv14v15}{rgb}{0.0,0.0,0.0}
\definecolor{cv14v20}{rgb}{0.0,0.0,0.0}
\definecolor{cv14v21}{rgb}{0.0,0.0,0.0}
\definecolor{cv14v0}{rgb}{0.0,0.0,0.0}
\definecolor{cv15v16}{rgb}{0.0,0.0,0.0}
\definecolor{cv15v1}{rgb}{0.0,0.0,0.0}
\definecolor{cv15v2}{rgb}{0.0,0.0,0.0}
\definecolor{cv15v3}{rgb}{0.0,0.0,0.0}
\definecolor{cv16v6}{rgb}{0.0,0.0,0.0}
\definecolor{cv16v7}{rgb}{0.0,0.0,0.0}
\definecolor{cv16v8}{rgb}{0.0,0.0,0.0}
\definecolor{cv16v17}{rgb}{0.0,0.0,0.0}
\definecolor{cv16v4}{rgb}{0.0,0.0,0.0}
\definecolor{cv16v5}{rgb}{0.0,0.0,0.0}
\definecolor{cv17v11}{rgb}{0.0,0.0,0.0}
\definecolor{cv17v18}{rgb}{0.0,0.0,0.0}
\Vertex[style={minimum size=1.0cm,draw=cv0,fill=cfv0,text=clv0,shape=circle},LabelOut=false,L=\hbox{$-1$},x=7.87cm,y=2.85cm]{v0}
\Vertex[style={minimum size=1.0cm,draw=cv1,fill=cfv1,text=clv1,shape=circle},LabelOut=false,L=\hbox{$1$},x=6.04cm,y=4.16cm]{v1}
\Vertex[style={minimum size=1.0cm,draw=cv2,fill=cfv2,text=clv2,shape=circle},LabelOut=false,L=\hbox{$1$},x=4.14cm,y=3.18cm]{v2}
\Vertex[style={minimum size=1.0cm,draw=cv3,fill=cfv3,text=clv3,shape=circle},LabelOut=false,L=\hbox{$1$},x=5.00cm,y=2.87cm]{v3}
\Vertex[style={minimum size=1.0cm,draw=cv4,fill=cfv4,text=clv4,shape=circle},LabelOut=false,L=\hbox{$-1$},x=3.7cm,y=6.11cm]{v4}
\Vertex[style={minimum size=1.0cm,draw=cv5,fill=cfv5,text=clv5,shape=circle},LabelOut=false,L=\hbox{$-1$},x=4.30cm,y=5.62cm]{v5}
\Vertex[style={minimum size=1.0cm,draw=cv6,fill=cfv6,text=clv6,shape=circle},LabelOut=false,L=\hbox{$-1$},x=2.36cm,y=5.85cm]{v6}
\Vertex[style={minimum size=1.0cm,draw=cv7,fill=cfv7,text=clv7,shape=circle},LabelOut=false,L=\hbox{$-1$},x=2.39cm,y=4.63cm]{v7}
\Vertex[style={minimum size=1.0cm,draw=cv8,fill=cfv8,text=clv8,shape=circle},LabelOut=false,L=\hbox{$1$},x=1.22cm,y=5.44cm]{v8}
\Vertex[style={minimum size=1.0cm,draw=cv9,fill=cfv9,text=clv9,shape=circle},LabelOut=false,L=\hbox{$1$},x=0.00cm,y=5.11cm]{v9}
\Vertex[style={minimum size=1.0cm,draw=cv10,fill=cfv10,text=clv10,shape=circle},LabelOut=false,L=\hbox{$1$},x=0.08cm,y=5.98cm]{v10}
\Vertex[style={minimum size=1.0cm,draw=cv11,fill=cfv11,text=clv11,shape=circle},LabelOut=false,L=\hbox{$1$},x=3.80cm,y=8.00cm]{v11}
\Vertex[style={minimum size=1.0cm,draw=cv12,fill=cfv12,text=clv12,shape=circle},LabelOut=false,L=\hbox{$1$},x=8.14cm,y=0.00cm]{v12}
\Vertex[style={minimum size=1.0cm,draw=cv13,fill=cfv13,text=clv13,shape=circle},LabelOut=false,L=\hbox{$1$},x=7.88cm,y=1.05cm]{v13}
\Vertex[style={minimum size=1.0cm,draw=cv14,fill=cfv14,text=clv14,shape=circle},LabelOut=false,L=\hbox{$-1$},x=6.78cm,y=2.39cm]{v14}
\Vertex[style={minimum size=1.0cm,draw=cv15,fill=cfv15,text=clv15,shape=circle},LabelOut=false,L=\hbox{$1$},x=4.99cm,y=3.75cm]{v15}
\Vertex[style={minimum size=1.0cm,draw=cv16,fill=cfv16,text=clv16,shape=circle},LabelOut=false,L=\hbox{$-1$},x=3.26cm,y=5.30cm]{v16}
\Vertex[style={minimum size=1.0cm,draw=cv17,fill=cfv17,text=clv17,shape=circle},LabelOut=false,L=\hbox{$1$},x=3.27cm,y=7.01cm]{v17}
\Vertex[style={minimum size=1.0cm,draw=cv18,fill=cfv18,text=clv18,shape=circle},LabelOut=false,L=\hbox{$1$},x=2.80cm,y=7.98cm]{v18}
\Vertex[style={minimum size=1.0cm,draw=cv19,fill=cfv19,text=clv19,shape=circle},LabelOut=false,L=\hbox{$1$},x=9.00cm,y=0.38cm]{v19}
\Vertex[style={minimum size=1.0cm,draw=cv20,fill=cfv20,text=clv20,shape=circle},LabelOut=false,L=\hbox{$-1$},x=6.56cm,y=1.49cm]{v20}
\Vertex[style={minimum size=1.0cm,draw=cv21,fill=cfv21,text=clv21,shape=circle},LabelOut=false,L=\hbox{$-1$},x=7.87cm,y=2.14cm]{v21}
\Edge[lw=0.05cm,style={color=cv8v9,},](v8)(v9)
\Edge[lw=0.05cm,style={color=cv8v10,},](v8)(v10)
\Edge[lw=0.05cm,style={color=cv12v13,},](v12)(v13)
\Edge[lw=0.05cm,style={color=cv13v14,},](v13)(v14)
\Edge[lw=0.05cm,style={color=cv13v19,},](v13)(v19)
\Edge[lw=0.05cm,style={color=cv14v15,},](v14)(v15)
\Edge[lw=0.05cm,style={color=cv14v20,},](v14)(v20)
\Edge[lw=0.05cm,style={color=cv14v21,},](v14)(v21)
\Edge[lw=0.05cm,style={color=cv14v0,},](v14)(v0)
\Edge[lw=0.05cm,style={color=cv15v16,},](v15)(v16)
\Edge[lw=0.05cm,style={color=cv15v1,},](v15)(v1)
\Edge[lw=0.05cm,style={color=cv15v2,},](v15)(v2)
\Edge[lw=0.05cm,style={color=cv15v3,},](v15)(v3)
\Edge[lw=0.05cm,style={color=cv16v6,},](v16)(v6)
\Edge[lw=0.05cm,style={color=cv16v7,},](v16)(v7)
\Edge[lw=0.05cm,style={color=cv16v8,},](v16)(v8)
\Edge[lw=0.05cm,style={color=cv16v17,},](v16)(v17)
\Edge[lw=0.05cm,style={color=cv16v4,},](v16)(v4)
\Edge[lw=0.05cm,style={color=cv16v5,},](v16)(v5)
\Edge[lw=0.05cm,style={color=cv17v11,},](v17)(v11)
\Edge[lw=0.05cm,style={color=cv17v18,},](v17)(v18)
\end{tikzpicture}
\caption{An $A$-bivalent tree associated with $\lambda_A=1.$}
\label{fig:Abivtree2}
\end{figure}


\noindent{\bf Acknowledgments.} The research of Leonardo de Lima is supported by CNPq grant 305988/2025-5.

\end{document}